\documentclass[12pt]{amsart}
\usepackage{amssymb, stmaryrd}
\usepackage{hyperref}
\usepackage[margin=1.3 in]{geometry}
\usepackage{color,soul}
\usepackage{multicol}
\usepackage{enumitem}
\usepackage[T1]{fontenc}

\theoremstyle{plain}
\newtheorem{theorem}[subsection]{Theorem}
\newtheorem{proposition}[subsection]{Proposition}
\newtheorem{lemma}[subsection]{Lemma}
\newtheorem{corollary}[subsection]{Corollary}

\theoremstyle{definition}
\newtheorem{definition}[subsubsection]{Definition}

\newtheorem{example}[subsubsection]{Example}

\newtheorem{remark}[subsection]{Remark}

\normalfont
\usepackage[T1]{fontenc}{}
\title{Classification of Frobenius Forms  \\ in five variables}

\author{Zhibek Kadyrsizova}
\address{Department of Mathematics, Nazarbayev University, Nur-Sultan, 010000, Kazakhstan}
\email{zhibek.kadyrsizova@nu.edu.kz}

\author{Janet Page} 
\address{Department of Mathematics, University of Michigan, Ann Arbor, MI 48109-1043, USA}
\email{jrpage@umich.edu}

\author{Jyoti Singh}
\address{Department of Mathematics,  Visvesvaraya National Institute of Technology,  Nagpur, Maharashtra 440010, India}
\email{jyotijagrati@gmail.com}

\author{Karen E. Smith}
\address{Department of Mathematics, University of Michigan, Ann Arbor, MI 48109-1043, USA}
\email{kesmith@umich.edu}

\author{Adela Vraciu}
\address{Department of Mathematics, University of South Carolina, Columbia, SC 29208, USA}
\email{vraciu@math.sc.edu}

\author{Emily E. Witt}
\address{Department of Mathematics, University of Kansas, Lawrence, KS 66045-7594, USA}
\email{witt@ku.edu}

\thanks{Partial support was provided  by  FDCRGP Grant 021220CRP5141 (for Zhibek Kadyrsizova), SERB(DST) Grant ECR/2017/000963 (for Jyoti Singh), 
NSF Grants \#1801697, \#1952399, and \#2101075
 (for Karen Smith),  and NSF CAREER Award \#1945611 (for Emily Witt),  as well as NSF Conference Grant DMS \#1934391 and AWM ADVANCE Grant NSF-HRD \#1500481}

\begin{document}

\begin{abstract}
We classify \emph{Frobenius forms}, a special class of homogeneous polynomials in characteristic $p>0$, in up to five variables over an algebraically closed field. We also point out some of the similarities with quadratic forms. 
\end{abstract}

\maketitle

\section{Introduction}

Fix a field $k$ of prime characteristic $p$.   
A Frobenius form over $k$ is a 
homogeneous  polynomial  in indeterminates $x_1, \dots, x_n$ of the form 
\begin{equation}\label{FF}
 x_1^{p^e} L_1 + x_2^{p^e}L_2 + \dots + x_n^{p^e}L_n
\end{equation}
where  each $L_i$ is some linear form and $e$ is a positive  integer.  
Put differently,  a Frobenius form is a homogeneous polynomial of degree $p^e+1$ that is  in the ideal generated by $x_1^{p^e}, \dots, x_n^{p^e}$.

 Frobenius forms can  be compared to quadratic forms:  if we allow $e=0$ in the expression (\ref{FF}) above, we get a quadratic form. 
Quadratic forms are  well-studied in the classical literature.  For example, much is known about the geometry of quadric hypersurfaces, 
and, at least over a quadratically closed field, their classification up to linear changes of coordinates is well-known. Both admit a convenient matrix factorization, where the action of changing coordinates behaves similarly.

One special but very interesting case of Frobenius forms are those of degree three (which are necessarily defined over a field of characteristic two). 
A smooth cubic surface  is always  Frobenius split, it turns out,  {\it unless} it is defined by a Frobenius form---in particular, non-Frobenius split (smooth)  cubic surfaces exist only over fields of characteristic two  \cite[5.5]{hara.rational-singularities}. A  detailed examination of non-Frobenius split cubic surfaces of characteristic two, including the non-smooth ones,   was undertaken in \cite{cubicspaper}.  
In particular, the Frobenius forms of degree three in up to four variables are classified there, up to projective equivalence.
 In this paper, we extend that classification to  arbitrary Frobenius forms in up to five variables. Put differently,  we classify the projective equivalence classes of   three-dimensional projective hypersurfaces defined by  Frobenius forms---a class called {\it extremal three-folds} in \cite{extremal-v1}. Section \ref{classification} describes this classification in detail,   including a particularly  "sparse" equation representing each of the seven types of projective equivalence classes of extremal three-folds.

Frobenius forms and the projective schemes they define have various  "extreme" properties, both algebraically and geometrically.
 For example,  cubic surfaces defined by  Frobenius forms are characterized by the geometric property that 
they "contain no triangles"---that is, any  plane section consisting of three lines {\it must} contain a point on all three \cite[5.1]{cubicspaper}.
Analogous  extremal  configurations of linear subvarieties occur more generally for extremal hypersurfaces of higher degree and dimension; see \cite[\S 8]{extremal-v1}. Algebraically, 
  reduced Frobenius forms can be characterized  as those  achieving the minimal possible $F$-pure threshold among  reduced forms of the same degree
 \cite[1.1]{extremal-v1}. 
 
 \medskip
 \noindent
  {\sc Acknowledgements.} This paper grew out of discussions begun at the AWM-sponsored workshop "Women in Commutative Algebra" at the Banff International Research Station in October 2019. We are grateful to Elo\'isa Grifo and Jennifer Kenkel, who also participated in those discussions.

 \section{Matrix Factorization of Frobenius Forms} \label{linear-algebra}
 
 Fix a field $k$ of prime characteristic $p$. Let $q$ denote an arbitrary positive power of $p$.
 
 \smallskip

The beauty of Frobenius forms is that, like quadratic forms,  they admit a matrix factorization:  a Frobenius  form $h$ in $n$ variables can be written as 
\begin{equation*}
h=  \begin{bmatrix}   
 x_1^{p^e} &  x_2^{p^e}  & \cdots & x_n^{p^e}
 \end{bmatrix}
 A 
\begin{bmatrix}  x_1 \\ x_2 \\ \vdots  \\ x_n 
\end{bmatrix}
 \end{equation*}
 for 
some  $n \times n$  matrix $A$ with entries in $k$. Because $e>0$, the matrix $A$ representing the Frobenius form $h$ is {\it unique}.

When $e = 0$, we recover the case of quadratic forms. In the quadratic form case, of course, $A$ is not unique, but we can force uniqueness, for example,  by insisting that $A$ be symmetric (when $p\neq 2$).  

There is an interesting story of Frobenius forms with "conjugate symmetric" matrices,    where conjugacy is defined using a Frobenius automorphism. 
By definition, a {\it Hermitian Frobenius  form} over  $k$ is a Frobenius form whose matrix $A$ satisfies $A^{\top} = A^{[q]}$ for some $q=p^e$ (in particular, such a matrix is defined over $\mathbb F_{q^2}$). This special class of Frobenius forms gives a characteristic $p$ analog of  Hermitian forms over the complex numbers, and has been studied, for example, in 
  \cite{segre}, \cite{BC},  \cite{HommaKim} and  \cite[\S35]{kollar14}. Like quadratic forms over quadratically closed fields, there is exactly one non-degenerate Hermitian Frobenius form, up to projective equivalence, in each dimension \cite[4.1]{BC} over $\mathbb F_{q^2}$. As we will see, the classification of more general Frobenius forms is more complicated.

To classify Frobenius forms up to projective equivalence, we need to understand how linear changes of coordinates act on them.

Let $g \in  GL_n(k)$ be any linear change of coordinates for the polynomial ring $k[x_1, \dots, x_n]$.
We represent the action of $g$ on the variables as 
 $$\begin{bmatrix}  x_1 \\ x_2 \\ \vdots  \\ x_n\end{bmatrix} \mapsto g  \begin{bmatrix}  x_1 \\ x_2 \\ \vdots  \\ x_n\end{bmatrix},$$
 the usual matrix multiplication. In particular, when $g$ acts\footnote{We write "$g  \cdot h$" to indicate the group action on polynomials, whereas adjacency, "$gB$," will indicate the usual matrix multiplication on an $n\times m$ matrix $B$.} on the Frobenius form $h$, we have
  \begin{align*}
g\cdot h &=  g \cdot \left( \begin{bmatrix}  x_1^{p^e} &  x_2^{p^e}  & \cdots & x_n^{p^e}\end{bmatrix}
 A \begin{bmatrix}  x_1 \\ x_2 \\ \vdots  \\ x_n\end{bmatrix} \right) 
&= \begin{bmatrix}  x_1^{p^e} &  x_2^{p^e}  & \cdots & x_n^{p^e}\end{bmatrix}(g^{[p^e]})^{\top} A g\begin{bmatrix}  x_1 \\ x_2 \\ \vdots  \\ x_n\end{bmatrix}, 
 \end{align*}
 where $g^{[p^e]}$ denotes the matrix whose entries are the $p^e$-th powers of the entries of $g$ and $B^\top$ indicates the transpose of a matrix $B$. 
 Thus, the action of $g$ on the  Frobenius form $h$ transforms the matrix $A$ representing  $h$ into the matrix  $(g^{[p^e]})^{\top} A g.$

A Frobenius form is said to be {\bf nondegenerate} if it cannot be written, after a linear change of coordinates, in a smaller number of variables, and is otherwise {\bf degenerate}, similarly as for quadratic forms. We define its {\bf embedding dimension} to be the smallest number of variables needed to write $f$ up to linear change of coordinates.  

The {\bf rank}  of a Frobenius form is defined as the rank of the representing matrix, just as the rank of  a quadratic form is the rank of the corresponding symmetric matrix when $p\neq 2$.  The rank of a Frobenius form is invariant under changes of coordinates, since the rank of a matrix is unchanged by  multiplication on the left and right by invertible matrices.

The rank of a Frobenius form is equal to the codimension of the singular locus of the corresponding hypersurface \cite[5.3]{extremal-v1}.  This is analogous to the
 corresponding statement about quadratic forms of non-even characteristic, when we work with the corresponding symmetric matrix.

 There is precisely  one full rank Frobenius form, up to change of coordinates, over an algebraically closed field, in each fixed degree $p^e+1$ and embedding dimension $n$:
 
\begin{theorem}\label{fullrank}  \cite[6.1]{extremal-v1}\cite{Beauville}
Every  full rank Frobenius form over an algebraically closed field $k$ of characteristic $p>0$ 
 is represented, in suitable linear coordinates,  by the   diagonal form  $x_1^{q+1}+ \cdots + x_n^{q+1},$ where 
$q$ is some power of $p$. 
\end{theorem}

 Put differently, every smooth projective hypersurface defined by a Frobenius form is projectively equivalent to one defined by $x_1^{q+1} + x_2^{q+1} + \dots + x_n^{q+1}. $ This is analogous to the situation for quadratic forms over a quadratically closed field of  characteristic not two, though the proof is a bit more involved. In particular, we do not have  a complete understanding of the situation over non-closed fields, an interesting open problem.  We have no counterexample to the speculation that the full rank Frobenius forms may be "diagonalizable" over a perfect field closed under all degree $p^e$ extensions. 

In the non-full rank case, there are finitely many projective equivalence classes of Frobenius forms in  each fixed  degree and embedding dimension  \cite[7.4]{extremal-v1}. Indeed,  the number of non-degenerate Frobenius forms of embedding dimension $n$  (of fixed degree) is bounded above by the $n$-th Fibonacci number.  However, the paper  \cite{extremal-v1} stopped short of precisely classifying the Frobenius forms in each dimension, a task essentially  completed  for Frobenius forms in four variables in \cite{cubicspaper}. 
The case of 
five variables is treated here in Section \ref{new}. For the statement, see Section \ref{classification}.

\section{Quadratic Forms}

To complete our story, we  recall the classification of quadratic forms:   
\begin{proposition}{\label{PropQuad}} 
Let $f$ be a non-degenerate  quadratic form in $n$ variables over a  quadratically  closed{\footnote{By quadratically closed, we mean that every degree two polynomial over $k$ splits. In characteristic two, this is a stronger assumption than requiring that the field contain the square root of every element.}}  field.
If the characteristic of $k$ is  two, then  $f$ is projectively equivalent to either
\begin{enumerate}
\item 
 $x_1x_2 + x_3x_4+ \dots + x_{n-2}x_{n-1}+ x^2_n$ if $n$ is odd, or 
\item   $x_1x_2 + x_3x_4+ \dots + x_{n-3}x_{n-2}+ x_{n-1}x_n$ if $n$ is even.
\end{enumerate}
If the characteristic of $k$ is not two, then   $f$   is projectively equivalent to  $$x_1^2+x_2^2+ \dots + x_n^2.$$
In particular, over a quadratically closed field, there is exactly one quadratic form in each embedding dimension.
\end{proposition}

  This classification is well known, but since we  could not find a low-tech proof in the modern literature for the case $p=2$, we include 
 one here.  Alternate discussions using more machinery  can be found, for example, in \cite[I.16]{Dieudonne},  \cite[12.9]{Grove} or  \cite[7.32]{EKM}, which also contain more refined classifications over non-closed fields.  A lower-tech proof can be found in the classic book of \cite{Dickson}, 
 but for the convenience of the reader we include a direct and concise proof here.{\footnote{While the stated classification of quadratic forms in characteristic two  is well known, it appears that an elementary proof is not. In his history of quadratic forms \cite{scharlau}, Scharlau laments that Dickson's 1899 work \cite{dickson-1899}, which is elementary but "rather involved,"  is not better known, stating   "However, one must admit, that this paper  - like  most of Dickson's work - is not very pleasant to read. It is  entirely algebraic." We hope the reader will find our straightforward and entirely algebraic proof more pleasant to read.}}

\begin{proof} The only quadratic form in one variable is $x_1^2$. Likewise, the two-variable case is trivial:  a degree two form in two variables must factor into two linear forms over a quadratically closed field, so in suitable coordinates,  the form is  either $x_1x_2$  or $x_1^2$ (which is degenerate).

\subsection*{Case of characteristic not two.} It is straightforward to check (even without closure assumptions on $k$) that
 a suitable choice of  linear change of coordinates  puts  $f$   in the form  $ \lambda_1x_1^2+\cdots+ \lambda_nx_n^2$, where  the  $\lambda_i$ are nonzero (e.g., see \cite{lam-quadratic-forms}).  So, if the ground field  is quadratically closed, 
 the change of coordinates taking each $x_i \mapsto \frac{1}{\sqrt \lambda_i} x_i$  normalizes the form to  $x_1^2+\cdots+x_n^2$. 

 \subsection*{Case of characteristic two.} 
Say that  $n\geq 3$. Since $f$ is non-degenerate, it is not the square of a linear form.
Thus some square-free term, which we can assume to be  $x_1x_2$,  appears with nonzero coefficient. Scaling, we may assume the coefficient of $x_1x_2$  is 1.
 
Now write $f$ in the form
\begin{equation}\label{temp1}
L^2 +  x_1 x_2 + \sum_{j=3}^n a_{1j} x_1 x_j + 
\sum_{j=3}^n a_{2j} x_2 x_j + h_1(x_3, \ldots, x_n).
\end{equation}
where $L$ is a (possibly zero) linear form in  $x_1, x_2$, and $h_1$ is a quadratic form in  $x_3, \ldots, x_n$.
Apply the linear change of coordinates sending $x_2$ to $x_2 + \sum_{j=3}^n {a_{1j}}x_j$, fixing the other variables.  This transforms 
(\ref{temp1}) into an expression which can be written  
\begin{equation}\label{eq2}
 L^2+ x_1 x_2 + \sum_{j=3}^n a'_{2j} x_2 x_j + h_2(x_3, \ldots, x_n),
\end{equation}
where again $h_2$ is a quadratic form in  $x_3, \ldots, x_n$.

With another  change of coordinates, we may assume  the summand $\sum_{j=3}^n a'_{2j} x_2 x_j$ is zero. 
Indeed, if some $a_{2j}'$ is nonzero, then we can assume $a_{23}'=1$ after renumbering and scaling if necessary.
Now  apply the linear transformation sending $x_3 $ to $ x_3 + \sum_{j=4}^n {a'_{2j}} x_j$, fixing the other variables.
This transforms the expression (\ref{eq2})  into 
\begin{equation}\label{quad-form3}
L^2 + (x_1 + x_3) x_2  + h_3(x_3, \ldots, x_n),
\end{equation}
where $h_3$ is quadratic  in  $x_3, \ldots, x_n$. Next, the coordinate change  taking $x_1$ to $x_1+x_3$, fixing the other variables, transforms (\ref{quad-form3}) into the  form
\begin{equation}\label{quad-form4}
 L^2 + x_1x_2  + h_4(x_3, \ldots, x_n)  
\end{equation}
with  $h_4$ quadratic.

Finally, by induction, we  separately apply linear changes of coordinates to  the quadratic $L^2 + x_1x_2$ in  $\{x_1, x_2\}$ and the quadratic $h_4$ in $\{x_3, \dots, x_n\}$ to put each into the desired form. It is  easy to see, then, that their sum has the desired form as well,  depending  on the parity of $n$ in the stated way. 
This completes the proof.
\end{proof}

\begin{remark}
Our  proof easily adapts to  show the well-known basic fact that a quadratic form over an arbitrary field of characteristic two is a sum of binary quadratics in distinct variables (plus a quadratic in one variable if the embedding dimension is odd). Alternatively, our proof adapts to prove Theorem 199 in \cite{Dickson} over any perfect field of characteristic two.
\end{remark}

\begin{remark}
Quadratic forms behave like Frobenius forms from the point of view of achieving the minimal $F$-pure threshold. For a reduced form of degree $d$, it is proved in \cite[1.1]{extremal-v1} that the $F$-pure threshold is at least $\frac{1}{d-1}$,  with equality if and only if the form is a Frobenius or quadratic form. Another way in which Frobenius forms and quadratic forms are similar is that the corresponding hypersurfaces both contain many high-dimensional linear subvarieties; see \cite[\S 8]{extremal-v1}.

\end{remark}

\section{Frobenius Forms in Five Variables: Classification Statements}\label{classification}

Fix an algebraically closed  field $k$ of positive characteristic $p$. Let $q$ be an arbitrary positive power of $p$.

\begin{theorem} \label{5} 
There are seven projective equivalence classes   of Frobenius forms of a fixed degree  $q+1$ and embedding dimension five. 
Specifically, these are represented by the following forms: 
\begin{enumerate}
 \item \, $\,x_1^{q+1}+x_2^{q+1}+x_3^{q+1}+x_4^{q+1}+x_5^{q+1}$\, (rank 5)
 \item \, $\, x_1^{q}x_5 +x_2^{q}x_4 + x_3^{q+1} + x_4^q x_2  $\, (rank 4)
 \item \, $ \, x_1^{q}x_5 +x_2^{q}x_4 + x_3^{q+1} + x_4^q x_1  $ \, (rank 4)
 \item \,  $\,  x_1^{q}x_5 +x_2^{q}x_4 + x_3^{q}x_2 + x_4^q x_1  $  \, (rank 4)
 \item \,  $\, x_1^{q}x_5 +x_2^{q}x_3 + x_3^{q}x_2 + x_4^q x_1  $  \, (rank 4)
 \item \,  $ \, x_1^{q}x_5 +x_2^{q}x_4 + x_3^{q+1}  $ \,  (rank 3)
 \item \,  $\, x_1^{q}x_5 +x_2^{q}x_4 + x_3^{q}x_2 $\,   (rank 3) 
\end{enumerate}
\end{theorem}

For completeness, we also describe  degenerate Frobenius forms  in  five variables in the following two theorems. The proofs are the same as in \cite{cubicspaper}, although that source considered only cubic Frobenius forms.

\begin{theorem} \label{4} 
There are five projective equivalence classes   of non-degenerate Frobenius forms of a fixed degree $q+1$  and embedding dimension four. 
 Specifically, these are represented by the following forms: 
 \begin{enumerate}
 \item \, $\, x_1^{q+1}+x_2^{q+1}+x_3^{q+1}+x_4^{q+1}$\, (rank 4)
 \item \, $\, x_1^q x_4+x_2^{q+1}+x_3^qx_1 $ \, (rank 3)
 \item \, $\, x_1^q x_4+ x_2^q x_3+ x_3^qx_1$ \, (rank 3)
 \item \, $\, x_1^q x_4+x_2^qx_3+x_3^qx_2$ \, (rank 3)
 \item \, $\, x_1^q x_4+x_2^q x_3$ \, (rank 2)
\end{enumerate}
\end{theorem}

\begin{theorem} \label{degenerate-cubics: T} 
There are three projective equivalence classes  of non-degenerate Frobenius forms  in three variables in each fixed degree, represented by precisely one of the following forms:
\begin{enumerate}
 \item \,   The  diagonal form $x_1^{q+1} + x_2^{q+1} + x_3^{q+1}$ \, (rank 3)
 \item \, The cuspidal form  $x_1^qx_3 + x_2^{q+1}$ \, (rank 2)
 \item \,  The reducible form  $x_1^q x_3 + x_2^qx_1 $ \, (rank 2)
 \end{enumerate}
 
Moreover,  in two variables,  each Frobenius form is projectively equivalently to exactly one of the following: 
\begin{enumerate}
 \item \,  The form   $x_1 x_2 (x_1^{q-1}+x_2^{q-1})$, defining $q+1$ distinct points  $0$, $\infty$ and the $(q-1)$-st roots of unity in $\mathbb P^1$. 
 \item \,  The form $x_1^q x_2$, defining the union of a $q$-fold point and a reduced point. 
 \item \,     The form $x_1^{q+1}$, defining a $(q+1)$-fold point.
\end{enumerate}
\end{theorem}

\section{Classification of Frobenius Forms  in Five Variables: Proofs} \label{new}

In this section, we prove Theorem \ref{5}. 

For this, we recall the method in \cite{extremal-v1} for  showing that there are only finitely many Frobenius forms  of any  fixed degree and embedding dimension, 
up to projective equivalence. A key point is that a Frobenius form  is equivalent to one represented by a   "{\bf sparse matrix}:"

\begin{theorem}\cite{extremal-v1}\footnote{A newer version of \cite{extremal-v1} proves the finiteness of Frobenius forms of a fixed degree and embedding dimension in a more uniform way, and no longer contains this result.  } \label{sparse}
Fix any algebraically closed field of characteristic $p>0$.  A Frobenius form of embedding dimension $n$ and rank $r$ can be represented by a matrix $A$  with the following properties:
\begin{enumerate}
\item All rows beyond the $r$-th are zero.
\item  All columns beyond the $n$-th are zero. 
\item There are exactly $r$ nonzero entries (all of which are 1) occurring in positions $(1 \, j_1), (2\,  j_2), \dots, (r \, j_r)$, where $j_1>j_2 >\dots > j_r$.
\end{enumerate}
\end{theorem}

In particular, we may assume that all columns of $A$ are zero but for $r$ of them, which are the standard unit basis vectors $e_r, \dots, e_1$ (in that order, and possibly interspersed with zero columns).  For example,  any  full rank Frobenius form is projectively equivalent to a Frobenius form whose matrix is the anti-diagonal matrix, that is,  whose columns are $e_n,\dots, e_1$.  A Frobenius form whose matrix satisfies the three conditions of Theorem \ref{sparse} will be called a {\bf sparse form.}

Theorem \ref{sparse}  implies the following bounds:

\begin{corollary}\label{rank}{} \

\begin{enumerate} 
\item
The rank of a Frobenius form of embedding dimension $n$ is at least $\frac{n}{2}$.
\item The number of  non-degenerate $n\times n$ matrices of rank $r$ satisfying the three conditions in  Theorem (\ref{sparse}) 
is $\binom{r}{n-r}$. \item  The number of projective equivalence types of Frobenius forms of  rank $r$ and embedding dimension $n$  (and fixed degree)  is at most  $\binom{r}{n-r}$.
\end{enumerate}

\end{corollary}

\begin{proof}
Choosing a sparse matrix to represent the Frobenius form, we can assume it looks like
$$
 x_1^ {p^e}L_1+ \cdots +  x_r^{p^e}L_r,
$$
where $L_1, \ldots, L_r \in \{x_1, \ldots, x_n\}$ are {\bf variables}  that appear in reverse order, each variable appearing at most once.
All  the  variables $x_{r+1}, \ldots, x_n$ must appear in the list $L_1, \ldots, L_r$ (otherwise the embedding dimension is less than $n$). In particular, 
$n-r\leq  r$.  This proves (1).

For (2), we continue by  observing that conditions (1) and (3) of sparseness (Theorem \ref{sparse}) force
$$
L_1=x_n,\,  L_2=x_{n-1}, \,  \ldots, \, L_{n-r}=x_{r+1}.
$$
 For   the  remaining linear forms $\{L_{n-r+1}, \ldots, L_r\}$, we can choose $2r-n$ variables out of  the  remaining variables  $\{x_1, \ldots, x_r\}$. There are 
$$
\binom{r}{2r-n} = \binom{r}{n-r}
$$
such
choices. So (3) follows as well.
\end{proof}

\medskip
It is easy to determine  the embedding dimension of a sparse Frobenius form: 

\begin{lemma} A rank $r$ Frobenius form of the type 
\begin{equation}\label{FFFF}
x_1^qx_{j_1} + x_2^qx_{j_2} + \dots + x_r^qx_{j_r} 
\end{equation}
 has embedding dimension equal to the number of distinct variables  appearing  in the expression (\ref{FFFF}).
\end{lemma}

\begin{proof} 
Suppose that a Frobenius form $f$ of the type (\ref{FFFF}) involves  $n$ variables and has rank $r$. If its embedding dimension is not $n$, then $f$ could be written as a polynomial in 
 $n-1$  independent linear forms, $y_1, \dots, y_{n-1}$. Without loss of generality, we can assume that the $y_i$ have the following  very special property: 
 there is an index $j$ such that  every $y_i$ is a {\it binomial} linear form  $x_{\ell_i} + a_{i} x_j$ for some index $\ell_i\neq j$ and scalar $a_{i}$. To see this, 
write
$$
\begin{bmatrix} y_1 \\ y_2 \\ \vdots \\ y_{n-1} \end{bmatrix} = B \begin{bmatrix} x_1 \\ x_2 \\ \vdots \\ x_n \end{bmatrix} 
$$
where $B$ is an $(n-1)\times n$ matrix of full rank, and  then left-multiply by the inverse in $GL(n-1)$ of a full rank  $(n-1)\times (n-1)$ submatrix of $B$. This replaces $\{y_1, \dots, y_{n-1}\}$ by  a set of  linear forms  spanning the same space and with the desired binomial form.

There are two cases to consider, depending on whether or not  $j \in \{1, \dots, r\}$.

If $j\leq r$, then without loss of generality $j=1,$ so that  $y_i = a_i x_1 + x_{i+1}$ for each $i=1, \dots, n-1$. Now if
$f$ is a Frobenius form in $y_1, \dots, y_{n-1}$, then 
\begin{equation}\label{aa}
 x_1^qx_{j_1} + x_2^qx_{j_2} + \dots + x_r^qx_{j_r} = y_1^qL_1 + y_2^qL_2 + \dots + y_{n-1}^qL_{n-1}
\end{equation}
for some linear forms $L_i$ in the $y_i$. Because the only terms on the right side of  (\ref{aa}) that involve $x_{r+1}^q, \dots, x_n^q$  come from $y^q_{r}, \dots, y^q_{n-1}$, respectively, it follows that 
$ L_{r} = \dots = L_{n-1} =0$. In this case, 
$$f = y_1^qL_1 + y_2^qL_2 + \dots + y_{r-1}^qL_{r-1},$$
so that $f$ has rank less than $r$, contrary to our hypothesis.

If $j>r$, then without loss of generality $j=n$, so that  $y_i =  x_i  + a_{i}x_n$ for each $i=1, \dots, n-1$. Now assuming
\begin{equation}\label{ab}
 x_1^qx_{j_1} + x_2^qx_{j_2} + \dots + x_r^qx_{j_r} = y_1^qL_1 + y_2^qL_2 + \dots + y_{n-1}^qL_{n-1}
\end{equation}
for some linear forms $L_i$ in the $y_j$,  again it follows that  $L_{r+1} = \dots = L_{n-1} =0$ by looking at the $x_i^q$ terms, with $i> r$. Moreover, since   the only terms on the right side of (\ref{ab})
 that involve $x_1^q, \dots, x_r^q$ come from $y^q_{1}, \dots, y^q_{r}$, we see that  $L_i=x_{j_i}$ for $i\leq r$. 
 But note that $x_n$ must appear among the variables $x_{j_1}, \dots, x_{j_r}$, since the original form $f$ involves all $n$ variables.  So $x_n$ is one of the $L_i$.
 This says that $x_n$ is a linear combination of $x_1+a_1x_n,  \, \dots, \, x_{n-1}+a_{n-1}x_n$, a contradiction. 
 
 Combining the two cases, we conclude that  $f$ is not a form in $y_1, \dots, y_{n-1}$, and so the embedding dimension of $f$ is $n$.
\end{proof}

\begin{remark} The total number of  projective equivalence classes of  Frobenius forms of embedding dimension $n$ is bounded  above by the $n$-th Fibonacci number  \cite[7.4]{extremal-v1}\footnote{In an updated version of \cite{extremal-v1}, the authors have recently given a precise count of the number of projective equivalence classes with a fixed embedding dimension.}. This follows from  Corollary \ref{rank} (3)  simply by adding the bounds for each relevant rank. This upper   bound  is sharp for $n \leq 4$, but not in general. 
 There are  distinct sparse matrices  that define equivalent Frobenius forms  starting in five  variables:
\end{remark}

\begin{example}\label{iso5by5}
Consider the Frobenius forms corresponding to the following matrices:
\[
\begin{bmatrix}
0 & 0&  0 & 0 & 1\\
0 & 0&  0 & 1 & 0\\
0 & 1&  0 & 0 & 0\\
0 & 0&  0 & 0 & 0\\
0 & 0&  0 & 0 & 0\\
\end{bmatrix}
\text{ and }
\begin{bmatrix}
0 & 0&  0 & 0 & 1\\
0 & 0&  0 & 1 & 0\\
1 & 0&  0 & 0 & 0\\
0 & 0&  0 & 0 & 0\\
0 & 0&  0 & 0 & 0\\
\end{bmatrix}
\]

\noindent namely, $x_1^qx_5 + x_2^qx_4 + x_3^qx_2$ and $x_1^qx_5 + x_2^qx_4 + x_3^qx_1$.  These are rank 3 nondegenerate Frobenius forms corresponding to distinct matrices satisfying the three conditions of Theorem \ref{sparse}, but they are equivalent by the change of coordinates which swaps $x_1 \leftrightarrow x_2$ and $x_4 \leftrightarrow x_5$, but fixes $x_3$.
\end{example}

The  proof of Theorem \ref{sparse} in \cite{extremal-v1}  relied on noting that a sparse $n \times n$ matrix of rank $r$ will take one of two forms
$$
\begin{bmatrix}
| &{\bf 0} & |\\
{\bf 0} & B & { e_1} \\
|& {\bf 0} & |\\
\end{bmatrix} 
\text{ or }
\begin{bmatrix}
| &{\bf 0} & |\\
{ e_r} & B & {e_1} \\
|& {\bf 0} & |\\
\end{bmatrix}
$$
where $B$ is a sparse matrix of rank $r-1$ in the first case, and a sparse matrix of rank $r-2$ in the second.  In light of this, we make the following definition.

\begin{definition}\label{def}
For $n \geq 3$, we say a sparse $n \times n$ matrix of rank $r < n$ has \emph{type a} if its first column is ${\bf 0}$, and \emph{type b} if its first column is ${e_r}$.
\end{definition}

In particular,  a Frobenius form of embedding dimension  $n$ and rank $r$ is equivalent, after a linear change of coordinates, to a polynomial of one of two forms: 
\begin{enumerate}[leftmargin=1.8cm]
 \item [({\bf type a})] $h = x_1^q x_n + h'(x_2, \ldots, x_{n-1})$
where $h'$ is a Frobenius form  of rank $r-1$ and embedding dimension $n-2$. In this case,  $B$ is non-degenerate of rank $r-1$.

\item [({\bf type b})] $h =  x_1^q x_n + x_r^q x_1 + h'(x_2, \ldots, x_{n-1})$
where $h'$ is a Frobenius form of rank $r-2$.
If $h'$ is non-degenerate in the $n-2$ variables $x_2, \dots x_{n-1}$, then $B$ is non-degenerate of rank $r-2$.  Otherwise, $h'$ must be non-degenerate in the $n-3$ variables $x_2, \ldots, \widehat{x}_r, \ldots, x_{n-1}$, so that $B$ is degenerate   of rank $r-2$. 
\end{enumerate}

\bigskip 

The proof of Theorem \ref{5}  will now follow from the following two propositions.

\begin{proposition}\label{n-1}
There are precisely $n-1$ projective equivalence classes of Frobenius forms of embedding dimension $n$ and rank $n-1$ (in fixed degree $p^e+1$).
\end{proposition}

The proof of Proposition \ref{n-1} uses the following lemma:

\begin{lemma}
Let $h_1$ and $h_2$ be two Frobenius forms of the same degree, both of embedding dimension  $n$ and rank $r= n-1$, and both  represented by matrices (say $A_1$ and $A_2$)  in sparse form.  Then $h_1$ and $h_2$ are projectively equivalent if and only if $A_1$ and $A_2$  are the same type,  and their $B$ matrices are projectively equivalent. \end{lemma}
\begin{proof}
Without loss of generality, by Theorem \ref{sparse} we may assume
\begin{align*}
 h_1 &= x_1^qx_n + x_2^qx_{j_2} + x_3^qx_{j_3} + \dots + x_{n-1}^qx_{j_{n-1}}, \text{ and } \\
h_2 &= x_1^qx_n + x_2^qx_{j_2'} + x_3^qx_{j_3'} + \dots + x_{n-1}^qx_{j_{n-1}'}
\end{align*}
where $x_{j_i}, x_{j_i'} \in \{x_1,\dots, x_{n-1}\}$ for each $i$.  In particular, we note that the only term containing $x_n$ in both $h_1$ and $h_2$ is $x_1^qx_n$.  Now suppose some change of coordinates $\phi$ sends $h_1$ to $h_2$.  The singular locus of both $h_1$ and $h_2$ is defined by the vanishing of  $x_1,   x_2, \dots, x_{n-1}$, so  the ideal $\left<x_1, \dots, x_{n-1}\right>$ is stable under  $\phi$.  In particular,  $\phi$  must send  $x_n $ to a linear form involving $x_n$. Supposing that  $\phi$ maps
\begin{align*}
 x_1 &\mapsto \lambda_{11}x_1 + \dots + \lambda_{1,n-1}x_{n-1}, \text{ and } \\
x_n &\mapsto \lambda_{n1}x_1 + \dots + \lambda_{nn}x_n\,\,\,\,\,\,{\text{with}} \,\,\lambda_{nn}\neq 0,
\end{align*}
then $\phi$ sends  $h_1$ to a polynomial
\begin{align*}
 &\lambda_{11}^q\lambda_{nn}x_1^qx_n  + \dots + \lambda_{1,n-1}^q\lambda_{nn}x_{n-1}^q x_n +  {\text{terms that do not involve }} x_n.
\end{align*}
In order for this to be equal to $h_2$, we must have  $\lambda_{1i} = 0$ for all $i > 1$. So $\phi$ sends
$$
x_1 \mapsto \lambda_{11}x_1
$$
where $\lambda_{11} \neq 0$.  In particular, if $h_1$ and $h_2$ are equivalent, they must be equivalent after modding out by $x_1$.

We can now see that if $h_1$ and $h_2$ are equivalent, then   $A_1$ and $A_2$ have the same type. 
Letting $\sim$ denote equivalence up to change of coordinates,  if 
$$
\begin{bmatrix}
0 & \bf 0 & 1 \\
\bf 0 & B &  \bf 0 \\
0 & \bf 0 &  0
\end{bmatrix} 
\sim 
\begin{bmatrix}
0 & \bf 0 & 1 \\
  e_{r-1} & B' &\bf 0 \\
 0 &  \bf  0 &  0
\end{bmatrix} 
$$
then $$\begin{bmatrix} B  & \bf 0 \\ \bf  0& 0\end{bmatrix} \sim \begin{bmatrix} B'  & \bf 0 \\  \bf 0 & 0\end{bmatrix}$$ which is the same as saying $B \sim B'$.  By rank considerations this never happens between type a and type b (see the discussion following Definition \ref{def}). 
Thus the sparse matrices of projectively equivalent Frobenius forms whose  rank is one less than the embedding dimension must have the same type. Furthermore, the argument above also shows that their "B" matrices are projectively equivalent.

Because the converse is obvious, the  lemma is proved. \end{proof}

\begin{proof}[Proof of Proposition \ref{n-1}]  Fix $q$. 
Let $N(n, r)$ denote the number of projective equivalence classes of Frobenius forms of degree $q+1,$ rank $r$, and embedding dimension $n$. 
We want to show that  $N(n,n-1) = n-1$. We will induce on $n$.

One readily verifies  that  
$$
N(1,0) = 0,  N(2,1) = 1, \text{ and } N(3,2) = 2.
$$
Since  type a and type b matrices yield distinct classes for $r = n-1$, the number of classes $N(n,n-1)$ is equal to the number of classes of type a plus the number of type b. 
The discussion  of the types following Definition  \ref{def} informs us, therefore, that for $n\geq 4$
$$
N(n,n-1) = N(n-2,n-2) + N(n-2,n-3) + N(n-3,n-3).
$$ 
Recalling that there is only one full rank form in each degree and dimension (Theorem \ref{fullrank}), it follows that
$$
N(n,n-1) = 1 + N(n-2,n-3)  + 1 = N(n-2,n-3)  + 2.
$$
Finally, by induction on $n$, 
$$
N(n,n-1) =  (n-3)  + 2 = n-1,
$$
as desired.
\end{proof}

\medskip

In light of Proposition \ref{n-1},  we have now fully classified all Frobenius forms in five variables of ranks four and five. 
 By Corollary \ref{rank} (1),  it only remains to analyze the rank three case. Using Corollary \ref{rank} (2), we see that there are $\binom{3}{2} = 3$ sparse forms of rank three in five variables, and we have seen in Example \ref{iso5by5} that two of them are projectively equivalent.  To complete the classification, it remains to see that the third sparse matrix produces a Frobenius form not equivalent to these. This is accomplished by the following: 

\begin{proposition}
The following rank three Frobenius forms in five variables are not projectively equivalent:
$$
f = x_1^qx_5 + x_2^qx_4 + x_3^{q+1}
$$
and 
$$
g = x_1^qx_5 + x_2^qx_4 + x_3^qx_2.
$$
\end{proposition}
\begin{proof}
To see this, note that $g \in \left<x_1,x_2\right>$.  Thus, if $f$ and $g$ are projectively equivalent, there must be some linear forms $L_1$ and $L_2$ such that $f \in \left<L_1,L_2\right>$.  Since any projective change of coordinates must respect the singular locus, any form in $x_1, x_2, x_3$ must be sent to another form in $x_1, x_2, x_3$, as 
the vanishing of these coordinates defines the singular set of both hypersurfaces. This implies that   $L_1,L_2$, being the  images of $x_1$ and $x_2$ under our linear change of coordinates, are forms in $x_1, x_2, x_3$.

 Now note that $\left< x_3^{q+1},x_4,x_5 \right>  = \left< f,x_4,x_5 \right> \subset \left<L_1,L_2, x_4,x_5\right>$.  In particular, $x_3 \in \sqrt{\left<L_1,L_2,x_4,x_5\right>} = \left<L_1,L_2,x_4,x_5\right>, 
$ so that $x_3 \in {\left<L_1,L_2\right>}$.  Without loss of generality, we may assume ${\left<L_1,L_2\right>} =  {\left<x_3,L_2\right>}$ 
where  $L_2$ is a linear form in $x_1$ and $x_2$, so that  $f \in \left<x_3,L_2\right>.$ Therefore,
\begin{equation}\label{AV}
x_1^qx_5 + x_2^qx_4 \in {\left<x_3,L_2\right>}
\end{equation}
as well.
Considering the image of the  expression (\ref{AV}) under the natural quotient map
 $k[x_1, \dots, x_5] \rightarrow k[x_1, \dots, x_5]/\langle x_3\rangle$, we see that
\begin{equation}\label{temp2}
x_1^qx_5 + x_2^qx_4  \in   {\left<L_2\right>} 
\end{equation} 
in a polynomial ring in four variables.   But  the polynomial $x_1^qx_5 + x_2^qx_4 $ is irreducible (for example, by Eisenstein's criterion), so we  arrive at a  contradiction. This contradiction ensures that $f,g$ are  not  projectively equivalent.
\end{proof}

This completes our classification of Frobenius forms in up to five variables.
{\small
\bibliographystyle{amsalpha}
\bibliography{bibdatabase}
}

\end{document}